\documentclass[a4paper,12pt]{amsart}
\usepackage{amsmath}
 \usepackage{latexsym, amsthm, amscd, euscript, float, subfig}
\usepackage{amssymb}
\usepackage{ifthen}
\usepackage{graphicx}
\usepackage{float}
\usepackage[utf8]{inputenc}
\usepackage{cite}
\usepackage{amsfonts}
\usepackage{amscd}
\usepackage{amsxtra}
\usepackage{color}
\usepackage{mathtools}

\newtheorem{theorem}{Theorem}[section]
\newtheorem{lemma}[theorem]{Lemma}
\newtheorem{corollary}[theorem]{Corollary}

\newtheorem{definition}[theorem]{Definition}

\theoremstyle{definition}
\newtheorem{remark}[theorem]{Remark}

\newtheorem{notation}[theorem]{Notation}

\numberwithin{equation}{section}

\def\be{\begin{equation}}
\def\ee{\end{equation}}

\def\dist{\text{ dist}}
\def\R{\mathbb{R}}
\def\ra{\rightarrow}
\def\cra{\curvearrowright}
\def\ds{\text{ ds}}
\def\dt{\text{ dt}}
\newcounter{alphabet}

\makeatletter
\newcommand*{\rom}[1]{\expandafter\@slowromancap\romannumeral #1@}
\makeatother

\begin{document}
\bibliographystyle{amsplain}
\title{Uniformization of intrinsic Gromov hyperbolic spaces\\ with busemann functions}
\author[Vasudevarao Allu]{Vasudevarao Allu}
\address{Vasudevarao Allu, School of Basic Sciences, 
Indian Institute of Technology Bhubaneswar,
Bhubaneswar-752050, 
Odisha, India.}
\email{avrao@iitbbs.ac.in}
\author[Alan P Jose]{Alan P Jose}
\address{Alan P Jose, School of Basic Sciences, Indian Institute of Technology Bhubaneswar,
Bhubaneswar-752050, Odisha, India.}
\email{alanpjose@gmail.com}

\begin{abstract}
For any intrinsic Gromov hyperbolic space we establish a Gehring-Hayman type theorem for conformally deformed spaces. As an application, we prove that any complete intrinsic hyperbolic space with atleast two points in the Gromov boundary can be uniformized by  densities induced by Busemann functions. Furthermore, we establish that there exists a natural identification of the Gromov boundary of $X$ with the metric boundary of the deformed space.
\end{abstract}

\maketitle

\section{Introduction}
Mikhael Gromov \cite{gromov_1987} demonstrated that the key asymptotic characteristics of the classical hyperbolic space $\mathbb{H}^{n}$ can be encapsulated in a simple inequality involving quadruples of points. 
This inequality was adequate to establish a theory of general hyperbolic spaces, known as Gromov hyperbolic spaces. 
These spaces capture many of the global features of classical hyperbolic space while  providing much greater flexibility.  
Gromov hyperbolic spaces represent a significant and well-researched class of metric spaces (see \cite{bridson}, \cite{lindquist}, \cite{herron_shanmu}), encompassing all complete, simply connected Riemannian manifolds with sectional curvature everywhere less than a negative constant. \\[2mm]
Bonk, Heinonen, and Koskela \cite{BHK} developed a uniformizing procedure for Gromov hyperbolic, geodesic and proper metric spaces, and obtained the following result:
\begin{theorem}\label{allu_jose_02_0001}
The conformal deformations $X_\epsilon$ of a proper, geodesic $\delta-$hyperbolic space $X$ are bounded $A(\delta)-$uniform spaces for $0<\epsilon\leq \epsilon_0$.
\end{theorem} 
A metric space is said to be uniform if any two points can be joined by a path which is neither too close to the boundary nor too crooked, see Definition \ref{allu_jose_02_0008} for a precise definition. 
Uniform domains were introduced independently by Martio and Sarvas \cite{martio_sarvas_} and Jones \cite{john_}, many  domains such as unit ball or upper half-plane in Euclidean domains serves as the examples of uniform domains. This class of domains became the 'nice domains' in several contexts including quasiconformal theory.\\[2mm]
Motivated by the work of Bonk et al., Butler \cite{butler} and Zhou et al \cite{zhou} independently developed unbounded counterparts of Theorem \ref{allu_jose_02_0001}  using a density that is exponential in a Busemann function. Busemann functions are known as the distance functions on a space from a point at infinity (see \cite{buyalo}). Zhou et al. considered proper geodesic $\delta-$hyperbolic spaces having atleast two points on the Gromov boundary whereas Butler \cite{butler} considered complete geodesic $\delta-$hyperbolic spaces which are roughly starlike from a point in the Gromov boundary. \\[2mm]
'Free quasiworld' introduced by J. V{\"a}is{\"a}l{\"a} is the extensive study of quasiconformal theory in infinite-dimensional Banach spaces where many of the conveniences of Euclidean spaces are absent (see \cite{quasiworld}). For instance, in this setting the space is not locally compact and he has demonstrated that existence of geodesics is not assured. 
He has also developed a foundational theory for more general hyperbolic spaces, which do not necessarily need to be geodesic or proper \cite{vaisala_2004} and thereby established the connections between hyperbolic and uniform domains in Banach spaces (see \cite{vaisala_2004_h}). \\[2mm]
Inspired from the works of Butler \cite{butler} and Zhou et al. \cite{zhou}, we enquire whether such a uniformization procedure with Busemann functions exists for intrinsic spaces which need not be geodesic or proper. 
A space is said to be intrinsic if the distance between any two points is equal to the infimum of the length of all curves connecting them.  
The main ingredient in the proof of Theorem \ref{allu_jose_02_0001} was the \textit{Gehring-Hayman Theorem}  \cite[Theorem 5.1]{BHK} for conformal deformations of hyperbolic spaces, which was also used by Butler \cite{butler}  and Zhou et al. \cite{zhou}.
The name is inspired  from the work of Gehring and Hayman \cite{gehring_hayman_}, which established that the hyperbolic geodesic within a simply connected hyperbolic domain in the plane minimizes the Euclidean length among all curves in the domain with same endpoints, up to a universal multiplicative constant. 
This theorem has been extended to quasihyperbolic geodesics in $\R^n$ that are quasiconformally equivalent to uniform domains (see \cite{hn}, \cite{hr}).\\[2mm]
Because we are concerned with spaces that  not necessarily geodesic, the Gehring-Hayman theorem established in \cite[Theorem 5.1]{BHK} does not apply. 
However, despite the absence of geodesics, we show that there exists a family of $h-$short arcs (defined in Section 2) which serves as a substitute for geodesics, enabling us to establish a theorem similar to Gehring-Hayman.
\begin{notation}
In this paper, $X$ denotes a metric space and $|x-y|$ is the distance between the points. By a curve $\gamma$ in $X$, we mean a continuous function $\gamma:[a, b] \rightarrow X$. An arc in a space $X$, is a subset homeomorphic to a real closed interval. We adopt the convention of using $\gamma$ to denote the parametrization of the curve as well as the image of the curve in $X$. To indicate that $\alpha $ is a curve or arc joining the points $x$ and $y$, we write $\alpha :x\cra y$. The length of a rectifiable curve $\gamma$ in a metric space is denoted by $l(\gamma)$.
\end{notation}
\begin{theorem}\label{allu_jose_02_0002}
Let $X$ be an intrinsic metric space which is $(\kappa, h)-$ Rips and $\rho :X\ra (0, \infty)$ be a continuous function that satisfies 
\be \label{allu_jose_02_0003}
\frac{1}{\lambda}\exp\{-\epsilon|x-y|\} \leq \frac{\rho(x)}{\rho(y)} \leq \lambda \exp\{\epsilon|x-y|\}
\ee
for all $x,y \in X$,  $\epsilon >0$ and for some fixed $\lambda\geq 1$. For $h<1/13$, let $\mathcal{A}$ denote the family of all $h-$short arcs $\alpha:x\cra y$ for any $x, y\in X$ satisfying $l(\alpha) \leq 2|x-y|$.
Then, there exists $\epsilon_0(\kappa, h)>0$, $K(\kappa, h)>0$ such that for any $0<\epsilon\leq \epsilon_0$  and $x, y\in X$ 
\begin{equation}\label{allu_jose_02_0004}
l_\rho(\alpha) \leq K l_\rho (\gamma)
\end{equation}
where  $\gamma: x\cra y $ is any curve and  $\alpha\in \mathcal{A}$ with same endpoints.
\end{theorem}
Let $X$ be a $\delta-$hyperbolic space. For any $\omega \in \delta_\infty X$ (see \ref{allu_jose_02_0011}), the function
\be 
b_\omega: X\times X \rightarrow \R,\,\,\, b_\omega(x, y) = (\omega | y)_x -(\omega |x )_y
\ee
is well defined. We let $\mathcal{B}(\omega) \subset \R^X $ be the set which consists of all functions $b: X\ra \R$ for each of which there are $o\in X$ and constant $c\in \R$ with $b(x) \dot{=} b_{\omega, o}(x)+c $, where $b_{\omega, o}(x) = b_\omega(x, o)$. 
Any function from $\mathcal{B}(\omega)$ is called a Busemann function based at $\omega \in\delta_\infty X$. 
We will mostly deal with the canonical functions $b_{\omega, o}$, where $o\in X$ and $\omega \in \delta_\infty X$.

We fix a point $\omega\in \delta_\infty X$ and  a Busemann function $b=b_{\omega, o}\in \mathcal{B}(\omega)$ with $o\in X$.  Define a family of maps $\rho_\epsilon :X \rightarrow (0, \infty)$ for $\epsilon >0$ as
\begin{equation}\label{allu_jose_02_0005}
\rho_\epsilon (x) = e^{-\epsilon b(x)}.
\end{equation}
Define a metric $d_\epsilon$ on $X$ by 
\begin{equation}\label{allu_jose_02_0006}
d_\epsilon(x, y) = \inf \int_\gamma \rho_\epsilon \ds ,
\end{equation}
where the infimum is taken over all rectifiable curves $\gamma $ joining $x$ and $y$. We write  $X_\epsilon = \left(X, d_\epsilon \right)$ for conformal deformation of $X$ with a conformal factor of $\rho_\epsilon$.                                                                                                                                                                                                                                                                                                                                                                                                                        Note that, for any rectifiable curve $\gamma$ in $X$, the length of the curve $\gamma$, denoted by $l_\epsilon\left(\gamma\right)$, in the metric space $X_\epsilon$ is given by
$$l_\epsilon(\gamma)= \int_\gamma \rho_\epsilon \ds,$$
provided the identity map $(X, |x-y|) \rightarrow (X, l)$ is a homeomorphism, where $l$ is the inner metric of $X$(see \cite[Lemma 2.6]{BHK} and Appendix \cite{BHK} for a detailed discussion).\\[2mm]
With the help of Theorem \ref{allu_jose_02_0002}, we prove that the deformed spaces using the conformal densities induced by Busemann functions are uniform.
\begin{theorem}\label{allu_jose_02_0007}
Let $X$ be a complete intrinsic $\delta-$hyperbolic space and suppose there are atleast two points in the Gromov boundary of $X$. Then the conformal deformations  $X_\epsilon = (X, d_\epsilon)$ are unbounded uniform spaces.
\end{theorem}
We also prove that there exists a natural identification of Gromov boundary of $X$ with the boundary of the deformed space $X_\epsilon$.
\begin{theorem}\label{allu_jose_02_0007.1}
Let $X$ be  a complete intrinsic $\delta-$hyperbolic space and $b_{\omega, o} \in \mathcal{B}(\omega)$ be a Busemann function. Suppose there exists atleast two points in the Gromov boundary of $X$. Then there exists a natural identification  $\phi$ between $ \delta_\infty X\setminus \{\omega\} $ and the boundary of the deformed space $X_\epsilon$.
\end{theorem}
\begin{remark}
Note that, even though Butler's work did not use any properness arguments, he considered geodesic metric spaces, whereas our work deals with spaces that are neither proper nor geodesic and his work relies heavily on tripod maps which are not available in non-geodesic metric spaces. 
\end{remark}
We realize in the study of non-proper Gromov hyperbolic spaces which are not necessarily geodesic some concepts  are more complicated than in the classical case, such as the center of an $h-$short triangle consists of three arcs instead of three points as in the classical case. But, as bounded distortions do not affect the structure of the space in large scale, we emphasize that the study of more general metric spaces is more suited to the philosophy of Gromov hyperbolic spaces.
\section{Preliminaries}
A metric space $X $ is said to be \textit{intrinsic }if 
$$|x-y|=\inf\left\{l(\alpha)|\alpha: x\cra y\right\},$$
for all $x, y\in X$. Let $h\geq 0$. An arc $\alpha:x\cra y$ is \textit{$h-$short} if
$$l(\alpha)\leq |x-y|+h.$$
\textit{Geodesics} are the isometric image of an interval in $\mathbb{R}$. We observe that $\alpha $ is a geodesic if and only if it is 0-short and a space $X$ is intrinsic if and only if for each pair $x, y\in X$ and for every $h>0$ there is an $h-$short arc connecting them. 
It is easy to see that every subarc of an $h-$short arc is $h-$short. 
For any curve $\gamma$, and for $u, v\in \gamma$, we use $\gamma[u,v]$ denotes the subcurve of $\gamma$ joining $u$ and $v$. For any rectifiable curve $\gamma:[a, b] \rightarrow X$, there is an arclength parametrization $\psi:[0, l(\gamma)] \rightarrow \gamma$ such that $l\left(\gamma(s, t)\right) = |s-t|$ for every $s,  t\in [0, l(\gamma)]$. We frequently use this fact in our proofs.
\begin{definition}\label{allu_jose_02_0008}
Let $X$ be an rectifiably connected, incomplete metric space. It is called $A-$ uniform if every pair of points $x, y\in X$ can be joined by a curve $\gamma$ in $X$ satisfying 
\be \label{allu_jose_02_0009}
l(\gamma) \leq A |x-y|
\ee
and
\be \label{allu_jose_02_0010}
\min\left\{l\left(\gamma[x, z]\right), l\left(\gamma[z, y]\right)\right\}\leq A \delta(z)
\ee
where $\delta(z)$ is the distance of $z$ to the boundary $\partial X = \bar{X}\setminus X$ of $X$.
\end{definition}
The condition \eqref{allu_jose_02_0009} is known as quasiconvexity of the curve and \eqref{allu_jose_02_0010} is known as the double cone arc condition.
We employ the standard notations of representing balls in $X$ as
$$B(a, r) = \{x\in X: |x-a|<r\} \text{ and } \bar{B}(a, r)= \{x\in X: |x-a|\leq r\}.$$
For a non empty set $A\subset X$, $N_r(A)$ be the $r-$neighbourhood of $A$ defined as 
$$N_r(A) = \{x\in X : d(x, A) <r\}.$$
\subsection{Gromov hyperbolic spaces}
\begin{definition}
For $x, y, p\in X$ we define the \textit{Gromov product} $(x|y)_p$ by 
$$2(x|y)_p = |x-p|+|y-p|-|x-y|.$$
Let $\delta \geq 0$. A space is \textit{Gromov $\delta-$hyperbolic} if 
$$(x|z)_p \geq \min\left\{(x|y)_p , (y|z)_p\right\} -\delta$$
for all $x, y, z, p\in X$. A space is \textit{hyperbolic} if it is Gromov $\delta-$hyperbolic for some $\delta\geq 0$.
\end{definition}
A classical example of a hyperbolic space is the Poincar\'{e} half space in $\R^n$ with the hyperbolic metric. 
In general,  V{\"a}is{\"a}l{\"a} \cite{vaisala_2004} has shown that uniform domains equipped with the quasihyperbolic metric in arbitrary Banach spaces are hyperbolic. Bounded spaces are trivially hyperbolic hence we make the assumption that every Gromov hyperbolic spaces we consider in this paper are unbounded.
\begin{definition}
Let $X$ be a metric space and let $o\in X$ be fixed. We say that a sequence $\{x_n\} $ in $X$ is a Gromov sequence if $(x_i|x_j)_o \rightarrow \infty$ as $i\rightarrow \infty$ and $j\rightarrow \infty$. 
\end{definition}
We say that two Gromov sequences $\{x_n\}$ and $\{y_n\}$ are equivalent if $(x_i|y_i)_o \rightarrow \infty $ as $i\rightarrow \infty$. In a hyperbolic space this is indeed an equivalence relation.
Let $\hat{x} $ be the equivalence class containing the Gromov sequence $\{x_n\} $ . The Gromov boundary $\delta_\infty X $ is defined as 
\be \label{allu_jose_02_0011}
\delta_\infty X = \left\{ \hat{x}: \{x_n\}  \text{ is a Gromov sequence in } X\right\}.
\ee
For $\zeta, \eta \in \delta_\infty X$ we define
$$(\zeta |\eta)_o = \inf \left\{\liminf_{i,j\rightarrow \infty} (x_i|y_j)_o: \{x_i\}\in \zeta, \{y_i\}\in \eta \right\}.$$
\begin{lemma}\label{allu_jose_02_0012}
Let $X$ be a $\delta-$hyperbolic space with $z, o\in X$ and $\zeta, \eta \in \delta_\infty X$. Then for $\{x_i\}\in \zeta$ and $\{y_i\}\in \eta$, we have 
$$(\zeta | \eta)_o \leq \liminf(x_i|y_i)_o \leq \limsup (x_i|y_i)_o \leq (\zeta | \eta)_o+2\delta ,$$
$$(\zeta | z)_o \leq \liminf(x_i|z)_o \leq \limsup (x_i|z)_o \leq (\zeta | z)_o+\delta .$$
\end{lemma}
We need the following result by  V{\"a}is{\"a}l{\"a} \cite[Tripod lemma 2.15]{vaisala_2004}.
\begin{lemma} \label{allu_jose_02_0013}
Suppose that $\alpha_i :a\cra b_i, i=1,2,$ are $h-$short arcs in a $\delta-$hyperbolic space. Let $x_1\in \alpha_1$ be a point with $|x_1 -a| \leq (b_1|b_2)_a$, and let $x_2, x_2' \in \alpha_2$ be points with $|x_2-a|=|x_1-a|$ and $l\left(\alpha_2[a, x_2']\right) = l\left(\alpha_1[a, x_1]\right)$. Then
$$|x_1-x_2| \leq 4\delta +h, |x_1-x_2'| \leq 4\delta+2h.$$
\end{lemma}
\begin{definition}[Triangles]\label{allu_jose_02_0014}
By an \textit{$h-$short triangle} we mean a triple of $h-$short arc $\alpha : y\cra z$, $\beta : x\cra z$, $\gamma: x\cra y$. The points $x, y, z$ are the vertices and the arcs $\alpha, \beta, \gamma$ are the sides of the triangle $\Delta=\left(\alpha, \beta, \gamma\right)$. 
\end{definition}
Following the notations used by V{\"a}is{\"a}l{\"a} \cite{vaisala_2004}, let $x_\gamma, y_\gamma$ be the points in $\gamma $ such that $l(\gamma_x)= (y|z)_x$ and $l(\gamma_y)= (x|z)_y$, where $\gamma_x = \gamma[x, x_\gamma]$ and $\gamma_y = \gamma[y_\gamma, y]$. If we set $\gamma^*= \gamma[x_\gamma, y_\gamma]$ which is called the center of the side $\gamma$ in the triangle $\Delta$, then $\gamma$ is the union of three subarcs $\gamma_x, \gamma^*, \gamma_y$. 

\begin{definition}[Rips condition]
Let $\delta \geq 0$. 
A triangle $\Delta$ in a space $X$ is $\delta-$slim if each side $\tau$ of $\Delta$ is contained in $\bar{B}\left(|\Delta|\setminus \tau, \delta\right)$, where $|\Delta|= \alpha \cup \beta\cup \gamma$. 
We say that an intrinsic space $X$ is $(\delta, h)-$Rips, if every $h-$short triangle in $X$ is $\delta-$slim. 
\end{definition}
We also need the following lemma in establishing the Gehring-Hayman theorem.
\begin{lemma}[Projection lemma]\cite[Lemma 3.6]{vaisala_2004}\label{allu_jose_02_0015}
Suppose $X$ is an intrinsic $(\kappa, h)-$Rips space. Let $\gamma\subset X$ be an $h-$short arc and let $x_1, x_2\in X$ and $y_1, y_2\in \gamma$ be points such that
\begin{enumerate}
\item $|x_i-y_i| = d\left(x_i, \gamma\right)\geq R >0$ for $i=1, 2,$
\item $|x_1-x_2| < 2R-4\kappa -h$.
\end{enumerate}
Then $|y_1-y_2| \leq 8\kappa +2h$.
\end{lemma}
In a proper, geodesic, hyperbolic space any point in the Gromov boundary and any point in $X$ can be joined by a geodesic ray. But in the case of general intrinsic hyperbolic spaces we do not have that luxury. V{\"a}is{\"a}l{\"a} \cite{vaisala_2004} introduced the concept of 'roads' to overcome this difficulty.
\begin{definition}
Let $X$ be a metric space and let $\mu \geq 0$, $h\geq 0$. A $(\mu, h)-$road in $X$ is a sequence $\bar{\alpha}$ of arcs $\alpha_i:y\cra u_i$ with the following properties.
\begin{enumerate}
\item Each $\alpha_i$ is $h-$short.
\item The sequence of lengths $l(\alpha_i)$ is increasing and tends to $\infty$.
\item For $i\leq j$, the length map $g_{ij}:\alpha_i \rightarrow \alpha_j$ with $g_{ij}y=y$ satisfies $|g_{ij}x-x|\leq \mu $ for all $x\in \alpha_i$.
\end{enumerate}
\end{definition}
It is possible to join any point in $X$ to any point in the boundary $\delta_\infty X$ by $(4\delta+2h, h)$-road.
\subsection{Busemann functions}
We refer the reader to \cite[Chapter 3]{buyalo} for a detailed discussion on the topic. 
\begin{notation}
Let $a, b, C$ be constants, following the notations in \cite{buyalo}, we write $a\dot{=}_C b$ if $|a-b|\leq C$. If $a, b>0$ and $C\geq 1$ then we write $a\asymp_C b $ instead of $1/C b \leq a \leq C b$. Also, we sometimes avoid mentioning the constant $C$ if it is not important, and just write $\dot{=}, \asymp$ instead of $\dot{=}_C, \asymp_C$, respectively.
\end{notation}
Every Busemann function $b \in \mathcal{B}(\omega)$ is roughly 1-Lipschitz, that is
\be \label{allu_jose_02_0016}
|b(x) - b(x')| \leq |x-x'|+ 10 \delta 
\ee
for every $x, x'\in X$. Also, for every $\omega \in\delta_\infty X$, every sequence $\{u_n\}\in \omega$ and every $o, x \in X$ it is known that (see \cite[Lemma 3.1.1]{buyalo})
\be \label{allu_jose_02_0017}
b_\omega(x, o) \dot{=}_{2\delta} (u_n|o)_x - (u_n|x)_o = |u_n-x|-|u_n-o|.
\ee
Applying \eqref{allu_jose_02_0016}, we derive the following Harnack's type inequality:
\begin{equation}\label{allu_jose_02_0018}
e^{-10\epsilon \delta}\exp\{-\epsilon|x-y|\} \leq \frac{\rho_\epsilon(x)}{\rho_\epsilon(y)}\leq e^{10\epsilon \delta}\exp\{\epsilon|x-y|\}
\end{equation}
Busemann functions can be used to define a Gromov product based at $\omega \in \delta_\infty X$. 
For a fixed Busemann function $b\in \mathcal{B}(\omega)$ and $x, y\in X$ the product is defined as 
\begin{equation*}
2(x|y)_b= b(x)+b(y)-|x-y|. 
\end{equation*}
Note that here $(x|y)_b$ can be negative, contrary to the standard case. Also, from the discussion in \cite[Section 3.2]{buyalo} we obtain that
\be \label{3.6}
(x|y)_b \dot{=}_{10\delta} (x|y)_o - (x|\omega)_o - (w|y)_o
\ee
for $b\in \mathcal{B}(\omega)$.
\section{Gehring-Hayman Theorem}
Let $X$ be an intrinsic, $(\kappa, h)-$Rips space, and let $\rho :X\ra (0, \infty)$ be a continuous function that satisfies \eqref{allu_jose_02_0003}
for all $x,y \in X$ and $\epsilon >0$. It is easy to verify that such a continuous function $\rho$ determines a metric space $(X, d_\rho),$ where the new metric is given by 
$$d_\rho(x, y) = \inf \int_\gamma \rho \ds;$$
and the infimum is taken over all rectifiable curves $\gamma$ joining $x$ and $y$. Also, if the identity map $\left(X, |.|\right) \ra \left(X, l\right)$, where $l(x, y)= \inf_\gamma l(\gamma)$, is a homeomorphism, then for each rectifiable curve $\gamma$ then length of $\gamma$  in the metric space $(X, d_\rho)$ denoted by $l_\rho$ is given by
$$l_\rho(\gamma)=\int_\gamma \rho \ds.$$
Note that, the $(\delta, h)$-Rips condition is quantitatively equivalent to $\delta-$hyperbolicity. Precisely, a $\delta-$hyperbolic space is $(\kappa, h)$ Rips with $\kappa= 3\delta+3h/2$ (see \cite[Theorem 2.34, Theorem 2.35]{vaisala_2004}. Hence, in this section we suppose that $X$ is $(\kappa, h)$-Rips and obtain the constants depending on $\kappa$ and $h$. For a $\delta-$hyperbolic space one can always obtain the constants in terms of $\delta$ by simply substituting $\kappa=3\delta+3h/2$.

The proof of Theorem \ref{allu_jose_02_0002} is accomplished through a series of lemmas. We employ techniques similar to those used by Bonk et al. \cite[Chapter 5]{BHK}. However, the lack of geodesics complicates the proofs significantly. We note that in the absence of geodesics, $h-$short arcs exhibit properties akin to geodesics if we select a sufficiently small $h>0$.
\begin{lemma}\label{allu_jose_02_0019}
Let $\gamma :x\cra y $ be a curve in $X$ satisfying 
\be \label{allu_jose_02_0020}
|x-y| \leq L=\frac{1}{24\lambda^2 \epsilon}
\ee
and 
\be \label{allu_jose_02_0021}
l(\gamma) \geq 6\lambda^2  |x-y|.
\ee
Then $l_\rho(\alpha) \leq l_\rho(\gamma)$ where $\alpha:x\cra y $ is such that $l(\alpha)\leq 2|x-y|$.
\end{lemma}
\begin{proof}
We compute,
\begin{align*}
l_\rho(\alpha) &=\int_\alpha \rho \ds\\
&\leq \lambda \rho(x)\int_0^{2|x-y|}e^{\epsilon t} \dt\\
& \leq \lambda \rho(x) \frac{1}{\epsilon} \left(e^{2\epsilon|x-y|}-1\right)
\end{align*}
Since, $e^x-1\leq \frac{3}{2}x $ whenever $0\leq x \leq \frac{1}{4}$ and  from \eqref{allu_jose_02_0020} we obtain
$$l_\rho(\alpha) \leq 3\lambda \rho(x)|x-y|.$$ 
Meanwhile, using \eqref{allu_jose_02_0021} we calculate
\begin{align*}
l_\rho(\gamma) &\geq \frac{\rho(x)}{\lambda}\int_0^{6\lambda^2|x-y|}e^{-\epsilon t}\dt\\
&=\frac{\rho(x)}{\lambda}\frac{1}{\epsilon} \left(1- e^{-\epsilon 6\lambda^2|x-y|}\right)\\
&\geq 4\lambda\rho(x)|x-y|.
\end{align*}
The last step follows from \eqref{allu_jose_02_0020} and the fact that $1-e^{-x} \geq \frac{2}{3}x $ whenever $0\leq x\leq \frac{1}{4}$. 
Hence the proof of Lemma \ref{allu_jose_02_0019} is complete.
\end{proof}
\begin{lemma}\label{allu_jose_02_0022}
Let $\bar{\gamma} :x\cra y $ be any rectifiable curve in $X$ then there is a curve $\gamma:x\cra y$ with $l_\rho(\gamma) \leq l_\rho(\bar{\gamma)}$ such that
\be \label{allu_jose_02_0023}
l\left(\gamma[u, v]\right) \leq 6\lambda^2 |u-v|+1
\ee
whenever $|u-v|\leq \frac{1}{24\lambda^2 \epsilon}$.
\end{lemma}
\begin{proof}
Let $\Gamma$ be the family of all curves in $X$ joining $x$ and $y$ such that $l_\rho(\gamma)\leq l_\rho(\bar{\gamma})$  for any $\gamma \in \Gamma$.
\\[2mm]
Clearly $\bar{\gamma} \in \Gamma$ and hence 
$$S=\inf_{\gamma \in\Gamma} l(\gamma) <\infty.$$
Choose a curve $\gamma\in \Gamma$ such that $l(\gamma)\leq S+\frac{1}{2}$. 
We claim that $\gamma$ is the required curve. If not, then there exists $u,v\in \gamma$ such that 
$$|u-v| < \frac{1}{24\lambda^2 \epsilon}$$
but
$$l(\gamma[u, v]) > 6\lambda^2|u-v|+1.$$
Then by Lemma \ref{allu_jose_02_0019}, $l_\rho (\alpha) \leq l_\rho (\gamma[u,v]),$ for a curve $\alpha:u\cra v$ such that $l(\alpha)\leq 2|u-v|$. \\[2mm]
Construct a new curve $\gamma' = \gamma[x,u] \cup \alpha \cup \gamma[v,y].$
It is easy to verify that $\gamma'\in \Gamma$ and hence we obtain 
\begin{align*}
S\leq l(\gamma') &= l(\gamma[x,u])+l(\alpha)+l(\gamma[v,y])\\
&=l(\gamma)-l(\gamma[u,v])+l(\alpha)\\
&\leq S+\frac{1}{2}-\left(6\lambda^2|u-v|+1\right)+2|u-v|\\
&\leq S-\frac{1}{2}
\end{align*}
which is the desired contradiction and the proof is complete.
\end{proof}
\begin{lemma}\label{allu_jose_02_0024}
Let $M,C,L$ be positive numbers and let $\gamma :x\cra y$ be a rectifiable curve satisfying
\be \label{allu_jose_02_0025}
l\left(\gamma[u,v]\right)\leq M|u-v|+C
\ee
whenever $|u-v|\leq L$. Let $0<h<\min \left\{1, \frac{1}{1+2M}\right\}$ and if for some $h-$short arc $\beta$ it holds that
\be \label{allu_jose_02_0026}
\dist\left(\gamma, \beta\right)=\dist\left(x, \beta\right) =\dist\left(y, \beta\right) \geq R 
\ee
where $R=1+4\kappa+4\kappa M+2h$, then either
\be \label{lem3eq3}
|x-y|>L
\ee
or
\be \label{lem3eq4}
|x-y| \leq A(\kappa, h, M, C).
\ee
\end{lemma}
\begin{proof}
Let $x_0, y_0 \in \beta$ be points such that 
$$|x-x_0|=|y-y_0|= \dist(\gamma, \beta).$$
Further, let $\beta_1:x\cra x_0$ and $\beta_2:y\cra y_0$ be $h-$short arcs.
By assumption \eqref{allu_jose_02_0026} we can choose points $x_1\in \beta_1$ and $y_1\in \beta_2$ such that 
$$|x-x_1|=R=|y-y_1|.$$\\
Again, let $\beta_3:x_1 \cra y_1$ be an $h-$short arc.\\[2mm]
We claim that 
\be \label{lem3eq5}
\dist(\gamma, \beta_3) \geq r:=R-2\kappa-2h = 1+2\kappa+4\kappa M.
\ee
If otherwise, there will exist points $z\in \gamma$ and $z_1\in \beta_3$ such that $|z-z_1|<r$. On one hand, by applying $(\kappa, h)-$Rips condition to the $h-$short rectangle $\beta_1[x_1,x_0]\cup\beta_2[y_1, y_0]\cup\beta_3\cup \beta[x_0, y_0]$, we obtain a point $z_0 \in  \beta_1[x_1,x_0]\cup\beta_2[y_1, y_0]\cup \beta[x_0, y_0]$ such that $|z_1-z_0|\leq 2\kappa.$ 
On the other hand, we will show that such a point cannot belong to any of these curves and thereby reach at a contradiction. \\[2mm]
If $z_0 \in \beta[x_0, y_0]$, then
$$R\leq |z_0-z| \leq |z_0-z_1| + |z_1-z| <r+2\kappa <R,$$
a contradiction.\\[2mm]
If $z_0\in \beta_1[x_1, x_0]$, we then  compute
\begin{align*}
|x-x_0| &= \dist(\gamma, \beta) \leq \dist(z, \beta)\\
&\leq |z-z_1|+|z_1-z_0|+|z_0-x_o|\\
&< r+2\kappa + |x_1-x_0|+h\\
&=R-h + |x_1-x_0|\\
&=|x-x_1|+|x_1-x_0|-h\\
&\leq l\left(\beta_1\right)-h\\
&\leq |x-x_0|,
\end{align*}
again a contradiction and hence $z_0\notin \beta_1[x_1, x_0]$. Similarly, one can show that $z_0\notin \beta_2[y_1, y_0]$. Thus the inequality \eqref{lem3eq5} follows.
\\[2mm]
Now, we find points $x=u_0, u_1,...,u_{n-1}, u_n=y$ in $\gamma$, where $n\geq 1$, such that
\begin{align*}
l\left(\gamma[u_k, u_{k+1}]\right) &= 1+8\kappa M-h, \text{ for } 0\leq k \leq n-2 \\
l\left(\gamma[u_{n-1}, u_n]\right) &\leq  1+8\kappa M-h .
\end{align*}
Clearly, \be \label{lem3eq5.1}
|u_k-u_{k+1}| \leq 1+8\kappa M < 2r-4\kappa -h
\ee
 since $h<1$.\\
Next, choose $v_k\in \beta_3 $ such that
\be \label{lem3eq6}
|u_k-v_k|=\dist (u_k,\beta_3).
\ee
Now,  \eqref{lem3eq5}, \eqref{lem3eq5.1}, \eqref{lem3eq6} together with projection lemma yields us
$$|v_k-v_{k+1}| \leq 8\kappa+2h$$
for $0\leq k\leq n-1$. Observe that
$$|u_0-v_0| = \dist (x, \beta_3)\leq |x-x_1| =R$$
and a similar argument gives that $|u_n-v_n|\leq R$.\\[2mm]
We are now ready to prove the lemma. Assume that $|x-y|\leq L$, then by \eqref{allu_jose_02_0025} we have
$$l(\gamma)\leq M|x-y|+C.$$
Thus,
$$n\leq \frac{l(\gamma)}{1+8\kappa M-h}+1 \leq \frac{ M|x-y|+C}{1+8\kappa M-h}+1 .$$
Hence,
\begin{align*}
|x-y|&=|u_0-u_n| \\
&\leq |u_0-v_0| +\sum_{k=0}^{n-1} |v_k-v_{k+1}| + |u_n-v_n|\\
&\leq 2R + n (8\kappa +2h)\\
&\leq 2R+(8\kappa +2h)\frac{ M|x-y|+C}{1+8\kappa M-h}+ 8\kappa +2h.
\end{align*}
An elementary computation gives us
$$|x-y| \leq  \frac{1}{1-h-2hM}\left(2R+8\kappa +2h\right)(1+8\kappa M-h) +\frac{(8\kappa+2h)C}{1-h-2hM},$$
we let 
\be\label{A}
A:=  \frac{1}{1-h-2hM}\left(2R+8\kappa +2h\right)(2+8\kappa M-h) +\frac{(8\kappa+2h)C}{1-h-2hM}
\ee 
then it follows that $|x-y|\leq A$ which is the desired result and the proof is complete.
\end{proof}
\begin{lemma}\label{ghlem4}
Let $\gamma:x\cra y $ be a rectifiable curve satisfying
\be \label{lem4eq1}
l(\gamma[u, v]) \leq M|u-v|+C
\ee
whenever $|u-v|\leq L:=2M( A +1)+1$, where 
$$A=\frac{1}{1-h-2hM}\left(2R+8\kappa +2h\right)(2+8\kappa M-h) +\frac{(8\kappa+2h)C}{1-h-2hM}$$ as in \eqref{A}.\\
Then, $\gamma \subset B\left(\alpha, L\right)$ for every $h-$short arc $\alpha: x \cra y$, provided $h< \min\left\{1, \frac{1}{1+2M}\right\}$.
\end{lemma}
\begin{proof}
We first assume that the curve $\gamma$ is parametrized by the arclength. That is, $\gamma: [0, l(\gamma)] \rightarrow X$ is such that $l\left(\gamma(s, t)\right) = |s-t|$ for every $0\leq s\leq t\leq l(\gamma)$.\\[2mm]
We define $f:[0, l(\gamma)]\rightarrow [0, \infty)$ such that
$f(t) =\dist \left(\gamma(t), \alpha\right).$ We claim that $f$ satisfies all the conditions of the \cite[Lemma 5.15]{BHK}.\\[2mm]
Since, both $\gamma$ and $\alpha$ have the same endpoints, it is evident that $f(0)=f(l(\gamma))=0$.\\
Also for $s,t\in [0, l(\gamma)]$, we have
$$|f(s)-f(t)|=|\dist\left(\gamma(s), \alpha\right) - \dist\left(\gamma(t), \alpha\right)|\leq |\gamma(s)-\gamma(t)| \leq l\left(\gamma(s, t)\right) =|s-t|.$$
We claim that $f$ satisfies the third condition with the constant $c= \frac{L}{2}$. Let $(s, t)\in M_c$, then for each $r\in[s, t]$
$$\dist \left(\gamma(r), \alpha\right) \geq  \dist \left(\gamma(s), \alpha\right)=\dist \left(\gamma(t), \alpha\right)\geq \frac{L}{2} \geq R,$$
where $R= 1+4\kappa+4\kappa M+2h$.\\[2mm]
Hence, by Lemma \ref{allu_jose_02_0024}, either $|u-v|>L$ or $|u-v|\leq A$, where $u=\gamma(s)$ and $v=\gamma(t)$.\\
If $|u-v|>L$, then 
$$2c=L <|\gamma(s)-\gamma(t)|\leq l\left(\gamma(s, t)\right)=t-s.$$
If $|u-v|\leq A$, then
$$|\gamma(s)-\gamma(t)| \leq A \leq L.$$
Together with \eqref{lem4eq1} yields us that
\begin{align*}
t-s=l\left(\gamma(s, t)\right) &\leq M.A+C \\
&=\frac{1}{2}(L-1) < \frac{L}{2}=c.
\end{align*}
Thus all the assumptions of \cite[Lemma 5.15]{BHK} are satisfied, and it implies that 
$$\max_{x\in [0, l(\gamma)]} f(x) <\frac{3}{2}c = \frac{3}{4}L <L$$
which completes the proof of the lemma.
\end{proof}
Now, we are prepared to proceed with the proof of the theorem.
\begin{proof}[Proof of Theorem \ref{allu_jose_02_0002}]
Choose $M=6\lambda^2, C=1$ in Lemma \ref{ghlem4} and let $L$ be the constant given in Lemma \ref{ghlem4}. 
Choose,
\begin{equation}\label{epsilonnot}
\epsilon_0= \frac{1}{25\lambda^2 L}.
\end{equation}
Let $\alpha : x\cra y $ be an $h-$ short arc  for $h < {1}/{13}$ satisfying $l(\alpha)\leq 2|x-y|$ and $\gamma :x\cra y$ be any other curve.
Note that, for each $h>0$ we can always select such a curve by choosing a curve $\alpha:x\cra y$ such that $l(\alpha) \leq c |x-y|$, for $c= \min \left\{2, 1+h/|x-y|\right\}$.
According to Lemma \ref{allu_jose_02_0022}, we can assume that $\gamma$ satisfies
\be 
l\left(\gamma[u, v]\right) \leq 6\lambda^2|u-v|+1
\ee
whenever $|u-v|\leq {1}/{24\lambda^2\epsilon}$.
Given the choice of $\epsilon_0$, this implies that
$$|u-v| \leq  \frac{25L}{24}.$$
Therefore,  $\gamma $ satisfies the conditions of Lemma \ref{ghlem4} with $M=6\lambda^2, C=1$ and thus $\gamma $ belongs to the $L-$neighbourhood of $\alpha$. Furthermore, by \cite[Lemma 3.5]{vaisala_2004}, $\alpha$ belongs to $2L+h-$ neighbourhood of $\gamma$, but clearly $2L+h \leq 3L$ and hence,
\begin{equation}\label{4lnbd}
\alpha\subset B\left(\gamma, 3L\right)
\end{equation}
We first consider the case when $|x-y| \leq 9L$.\\[2mm]
On one hand we have,
\begin{align*}
l_\rho(\alpha) &= \int_\alpha \rho \ds \leq \int_\alpha \lambda\rho(x) e^{\epsilon_0 |x-\alpha(t)|} \dt\\
&\leq \int_\alpha \lambda\rho(x) e^{\epsilon_0 l\left(\alpha[x, \alpha(t)]\right)} \dt\\
&\leq \int_\alpha \lambda\rho(x) e^{\epsilon_0 18L}\dt\\
&\leq \lambda \rho(x) e^{\epsilon_0 18L} l(\alpha)\\
&\leq 2\lambda \rho(x) e^{\epsilon_0 18L} |x-y|.
\end{align*}
On the other hand, there exists a subcurve $\gamma'$ of the curve $\gamma: x\cra y$ emanating from $x$ such that  $l(\gamma')=|x-y|$. Therefore, we compute
\begin{align*}
l_\rho(\gamma)\geq l_\rho(\gamma') &= \int_{\gamma'} \rho \ds \geq \int_{\gamma'} \frac{1}{\lambda}\rho(x) e^{-\epsilon_0|x-\gamma(t)|}\dt\\
&\geq \int_{\gamma'} \frac{1}{\lambda} \rho(x) e^{-\epsilon_0|x-\gamma'(t)|}\dt\\
&\geq \int_{\gamma'} \frac{1}{\lambda} \rho(x) e^{-\epsilon_0 |x-y|}\dt\\
&\geq \frac{1}{\lambda} \rho(x) e^{-\epsilon_0 9L} |x-y|.
\end{align*}
It follows that in this case,
$$l_\rho(\alpha) \leq 2\lambda^2 e^{\epsilon_0 27L} l_\rho(\gamma).$$
Thus with the choice of $\epsilon_0 $ as in \eqref{epsilonnot}, we see that, in this case we can choose the constant $K= 6\lambda^2$.\\[2mm]
Now, if $|x-y|\geq 9L$, then choose points $x=x_1, x_2,...,x_n, x_{n+1}=y, n\geq 2$ on the curve $\alpha$ such that 
$l\left(\alpha[x_k, x_{k+1}]\right)= 9L$ and $l\left(\alpha[x_n, x_{n+1}]\right)< 9L.$ Let $\alpha_k= \alpha[x_k, x_{k+1}]$.\\[2mm]
We claim that $B(x_k, 4L)\cap B(x_m, 4L)=\emptyset$ for  $1\leq k<m \leq n$.
If this is not the case, then for $z\in B(x_k, 4L)\cap B(x_m, 4L)$, we obtain on one hand
$$|x_k-x_m| \leq |x_k-z|+|x_m-z|<8L,$$
and on the other hand,
$$|x_k-x_m|\geq l\left(\alpha[x_k, x_m]\right)-h \geq 9L -h >8L,$$
which is the desired contradiction, and thus the claim holds. \\[2mm]
By \eqref{4lnbd}, $\gamma$ intersects the balls $B(x_k,3L)$ for every $k$, and since $n\geq 2$, $\gamma$ cannot be entirely within $B(x_k, 4L)$ for any $k$.
 Consequently, there exists a subcurve $\gamma_k$ of $\gamma$ contained in $\bar{B}(x_k, 4L)\setminus B(x_k, 3L)$ connecting $B(x_k, 3L)$ to $X\setminus B(x_k, 4L)$ for $1\leq k \leq n$. 
 Additionally, these subcurves are mutually disjoint because the family  $\left\{B(x_k, 4L), 1\leq k \leq n\right\}$ is.
Therefore, 
$$l_\rho(\gamma)\geq \sum_{k=1}^n l_\rho \left(\gamma_k\right).$$
Now for $1\leq k\leq n$ we have,
\begin{align*}
l_\rho(\gamma_k) &= \int_{\gamma_k} \rho \ds \geq \int_{\gamma_k}\frac{1}{\lambda}\rho(x_k) e^{-\epsilon |x_k-\gamma_k(t)|}\dt\\
&\geq \int_{\gamma_k}\frac{1}{\lambda}\rho(x_k)e^{-\epsilon_0 4L}\dt\\
&\geq \frac{1}{\lambda}\rho(x_k)e^{-\epsilon_0 4L} L
\end{align*}
and
\begin{align*}
l_\rho(\alpha_k) &= \int_{\alpha_k} \rho \ds \leq \int_{\alpha_k} \lambda\rho(x_k) e^{\epsilon |x_k-\alpha_k(t)|}\dt\\
&\leq \int_{\alpha_k} \lambda\rho(x_k) e^{\epsilon_0 9L}\dt\\
&\leq \lambda\rho(x_k) e^{\epsilon_0 9L}9L.
\end{align*}
Therefore,
$$l_\rho(\alpha) = \sum_{k=1}^n l_\rho(\alpha_k) \leq \sum_{k=1}^n\lambda\rho(x_k) e^{\epsilon_0 9L}9L \leq 9\lambda^2 e^{\epsilon_0 13L}\sum_{k=1}^nl_\rho(\gamma_k)\leq 9\lambda^2 e^{\epsilon_0 13L} l_\rho(\gamma). $$
A simple estimation shows that here you can take the constant $K=18\lambda^2$.
Hence the proof of the Theorem \ref{allu_jose_02_0002} is complete.
\end{proof}
\section{Uniformization}
Before proving that the deformed spaces are indeed uniform, we need the following two crucial lemmas.
\begin{lemma}\label{lem1}
Let $\gamma:x\cra y$ be an $h-$short arc in the metric space $X$ and $p\in X$ be fixed. Let $x_\gamma, y_\gamma\in \gamma$ be points in the curve $\gamma$ such that $l\left(\gamma[x, x_\gamma] \right)= (y|p)_x$ and $l\left(\gamma[y_\gamma, y] \right)= (x|p)_y$. If $z\in \gamma[x, y_\gamma]$ and $u\in \gamma[x, z]$, then
$$|p-u|-|p-z|\geq |u-z|-8\delta-8h.$$
\end{lemma}
\begin{proof}
Construct an $h-$short triangle $\Delta$ with sides $\alpha, \beta, $ and $\gamma$, where $\alpha:x\cra p$ and $\beta : y\cra p$. We follow the notations as in \ref{allu_jose_02_0014}.
First, suppose that $z\in \gamma_x$ and $u\in \gamma[x, z]$. Then, there exist points $z', u'\in \alpha_x$ such that $l\left(\alpha[x, z']\right) = l\left(\gamma[x,z]\right)$ and $l\left(\alpha[x, u']\right) = l\left(\gamma[x, u]\right).$ Hence, by tripod lemma \ref{allu_jose_02_0013} $|z-z'|\leq 4\delta+2h$ and $|u-u'|\leq 4\delta +2h$.\\[2mm]
Now, we compute,
\begin{align*}
|p-u|-|p-z| &\geq |p-u'|-|u-u'|-|p-z'|-|z-z'| \\
&\geq l\left(\alpha[u', p]\right)-h -l\left(\alpha[z', p]\right)-8\delta-4h\\
&=l\left(\alpha[u', z']\right)-8\delta-5h\\
&=l\left(\alpha[u, z]\right)-8\delta-5h\\
&\geq |u-z|-8\delta-5h
\end{align*}
Now, if $z\in \gamma[x_\gamma, y_\gamma]$ and $u\in \gamma[x, x_\gamma]$, then there exists $u'\in \alpha[x, x_\alpha]$ such that $|u-u'|\leq 4\delta +2h$.\\[2mm]
We observe that,
\begin{equation*}
l\left(\gamma[x, u]\right)+l\left(\gamma[u,x_\gamma]\right) = l\left(\gamma[x, x_\gamma]\right)  = l\left(\alpha[x, x_\alpha]\right) \leq |x-x_\alpha|+h\leq l\left(\alpha[x, u']\right) +|u'-x_\alpha|+h
\end{equation*}
which implies that
\be \label{eq0}
l\left(\gamma[u,x_\gamma]\right) \leq l\left(\alpha[u', x_\alpha]\right)+h.
\ee
Thus,
\begin{align*}
|p-u|-|p-z|&\geq |p-u'|-|u-u'|-|p-x_\alpha|-|x_\alpha - z|\\
&\geq l\left(\alpha[u', p]\right)-h-l\left(\alpha[x_\alpha, p]\right)-|x_\alpha -z|-4\delta-2h\\
&\geq l\left(\gamma[u,x_\gamma]\right) -|x_\alpha -x_\gamma|-|x_\gamma-z|-4\delta-4h\\
&\geq l\left(\gamma[u, z]\right)- l\left(\gamma[x_\gamma, z]\right)-8\delta-8h\\
&\geq |u-z|-8\delta-8h,
\end{align*}
where we used \eqref{eq0} in the third step.\\[2mm]
Finally we consider the case where $u, z\in \gamma[x_\gamma, y_\gamma]$ and compute
\begin{align*}
|p-u|-|p-z| &\geq |p-x_\alpha|-|x_\alpha - u|- |p-p_\alpha|-|p_\alpha -z|\\
&\geq l\left(\alpha[p, x_\alpha]\right) -h - l\left(\alpha[p, x_\alpha]\right) -|x_\alpha-u| -|p_\alpha -z|\\
&\geq -8\delta-7h\\
&\geq |u-z|-8\delta-8h
\end{align*}
where in the last step we used the fact that $|u-z|\leq h$.
\end{proof}

\begin{lemma} \label{conelemma}
Let $X$ be a $\delta-$hyperbolic space and let $x, y\in X$, $\omega\in \delta_\infty X$. 
If $\gamma : x\cra y $ is any $h-$short arc and $b=b_{\omega, o} \in \mathcal{B}(\omega)$, then
 $$b(u)-b(z) \geq |u-z|-16\delta-10h, \text{ for any }z\in \gamma[x, y'] \text{ and for all } u \in \gamma[x, z],$$
 where $y' \in \gamma $ is such that $l(\gamma[y', y]) = (x|\omega)_y$. 
\end{lemma}
\begin{proof}
Let $x'\in \gamma $ be such that $l(\gamma[x, x']) = (y|\omega)_x$, then it follows from Lemma \ref{allu_jose_02_0012} that 
$$l(\gamma[x, x'])+ l(\gamma[y', y]) = (y|\omega)_x + (x|\omega)_y \leq |x-y| \leq l(\gamma),$$
which implies that $x'\in \gamma[x, y']$.

Moreover, by Lemma \ref{allu_jose_02_0012} we obtain
$$l(\gamma[x', y']) \leq l(\gamma)-(x|\omega)_y-(y|\omega)_x \leq |x-y|+h -(x|\omega)_y-(y|\omega)_x \leq 2\delta +h, $$
which yields us 
$$|x'-y'|\leq 2\delta +h.$$
Let $\bar{\alpha} : x\cra \omega $ be a $(4\delta+2h, h)-$road. Write $\alpha_n:x\cra u_n$, then it is known that the sequence $\{u_n\}\in \omega$. 
Now, pick points $x_n, y_n\in \gamma$ satisfying $l(\gamma[x, x_n]) = (y|u_n)_x$ and $l(\gamma[y_n, y]) = (x|u_n)_y$. 
Further, for each $r>0$,  by virtue of Lemma \ref{allu_jose_02_0012}, we have
\be \label{eq1}
l(\gamma[x, x']) = (y|\omega)_x \leq (y|u_n)_x +\frac{1}{2} = l(\gamma[x, x_n]) +r,
\ee
and 
\be \label{eq2}
l(\gamma[y', y]) =(x|\omega)_y \leq (x|u_n)_y +\frac{1}{2} = l(\gamma[y_n, y]) +r
\ee
for sufficiently large $n$.
Let $x'',y''\in\gamma $ be such that $l(\gamma[x, x'']) = (y|\omega)_x -r =  l(\gamma[x, x']) - r$ and $l(\gamma[y'', y])  =(x|\omega)_y -r = l(\gamma[y', y])-r$, provided both $(y|\omega)_x -r$ and $(x|\omega)_y -r$ are non-negative numbers, otherwise we let $x''=x$ and $y''=y$. Hence, it follows from \eqref{eq1}, \eqref{eq2} and the fact that $x_n\in \gamma[x, y_n]$ that $x_n, y_n \in \gamma[x'', y'']$. Also, note that $l(\gamma[x'', y''])= 2r+l(\gamma[x', y'])\leq 2r+2\delta+h$.\\
We claim that \be \label{rineq} 
|u_n-u|-|u_n-z| \geq |u-z|-2r-12\delta-10h.
\ee
We consider three cases:\\
\textit{Case 1:} Assume that $z\in \gamma[x, y_n]$, then for $u\in \gamma[x, z]$ by Lemma \ref{lem1} we have 
$$|u_n-u|-|u_n-z|\geq |u-z|-8\delta-8h.$$
\textit{Case 2:} If $z, u \in \gamma[y_n, y']$, then $|u-z| \leq l(\gamma[x'', y''])\leq 2r+2\delta+h.$ \\
Therefore, we obtain,
$$|u_n-u|-|u_n-z|\geq -|u-z| \geq |u-z|-4\delta-2h-2r.$$
\textit{Case 3:} Suppose that $z\in \gamma[y_n, y']$ and $u\in \gamma[x, y_n]$ and by case 1, we have $|u_n-u| -|u_n-y_n| \geq |u-y_n|-8\delta -8h$.  Then, we compute
\begin{align*}
|u_n-u|-|u_n-z| &\geq |u_n-u| -|u_n-y_n|-|y_n-z|\\
&\geq |u-y_n|-|y_n-z|-8\delta -8h\\
&\geq |u-z|-2|y_n-z|-8\delta -8h\\
&\geq |u-z|-2r-12\delta-10h.
\end{align*}
Hence, our claim is true and the lemma follows from \eqref{allu_jose_02_0017}.
\end{proof}
\begin{lemma} \label{comparelemma}
Let $x, y\in X$ and $\gamma : x\cra y $ be an $h-$short arc, satisfying $l(\gamma) \leq 2|x-y|$. Then there exists a constant $C=C(\delta, \epsilon, h)$ such that 
\be \label{comparelemmaeq}
l_\epsilon(\gamma) \asymp_C e^{-\epsilon(x|y)_b} \min \left\{\frac{1}{2}, \epsilon|x-y|\right\}.
\ee
\end{lemma}
\begin{proof}
Let $\gamma : x\cra y$ be an $h-$short arc satisfying $l(\gamma) \leq 2|x-y|$. Further, let $y'\in \gamma$ be such that $l(\gamma[y', y])=(x|\omega)_y$. Then, by Lemma \ref{conelemma} and from \eqref{allu_jose_02_0016}, we have
$$
 |x-y'|-16\delta-10h \leq b(x)-b(y')\leq |x-y'|+10\delta
$$
and by symmetry we have
$$ 
|y-y'|-16\delta-10h \leq b(y)-b(y')\leq |y-y'|+10\delta ,
$$
which yields us
$$ 
|x-y'|+|y-y'|-32\delta-20h \leq b(x)+b(y)-2b(y') \leq |x-y'|+|y-y'|+20\delta.
$$
Thus, a simple computation gives us 
$$
|x-y|-32\delta-23h \leq b(x)+b(y)-2b(y') \leq |x-y| +20\delta +h,
$$
which implies that 
$$ 
-16\delta-\frac{23}{2}h  \leq (x|y)_b-b(y')\leq 10\delta+\frac{h}{2}.
$$
Thus we obtain,
\be \label{b(y')}
b(y')\dot{=}_{16\delta+12h} (x|y)_b.
\ee
Hence, it suffices to show that there exists $C'=C'(\delta, \epsilon, h)$ such that
\be \label{inter} 
l_\epsilon(\gamma) \asymp_{C'} \rho_\epsilon(y') \min \left\{\frac{1}{2}, \epsilon|x-y|\right\}.
\ee
If $\epsilon |x-y| \leq 1/2$, then for all $u\in \gamma$, we have
$$e^{\epsilon |u-y'|}\leq e^{\epsilon l(\gamma [u, y'])} \leq e^{\epsilon l(\gamma)} \leq e.$$
Together with \eqref{allu_jose_02_0018}, we obtain
\be 
\frac{1}{e^{10\epsilon\delta +1} } \rho_\epsilon(y')\leq \frac{e^{-\epsilon |u-y'|}}{e^{10\epsilon\delta}} \rho_\epsilon(y') \leq \rho_\epsilon(u) \leq e^{10\epsilon\delta} e^{\epsilon |u-y'|} \rho_\epsilon(y') \leq e^{10\epsilon\delta +1}  \rho_\epsilon(y').
\ee
Therefore, on one hand we have
$$l_\epsilon (\gamma)= \int_\gamma \rho_\epsilon(u) |du| \leq 2 e^{10\epsilon\delta +1}  \rho_\epsilon(y')|x-y| $$
and on the other hand, we have 
$$l_\epsilon(\gamma) = \int_\gamma \rho_\epsilon |du| \geq \frac{\rho_\epsilon(y')}{ e^{10\epsilon \delta +1}} |x-y|.
$$
Now, if $\epsilon |x-y| > 1/2$, then by Lemma \ref{conelemma} for every $u\in \gamma$, we have 
$$e^{-\epsilon b(u) } \leq e^{-\epsilon b(y')}e^{-\epsilon |u-y'|} e^{\epsilon 16\delta+\epsilon 10h},$$
which gives us 
$$\rho_\epsilon(u) \leq e^{\epsilon 16\delta+\epsilon 10h} \rho_\epsilon(y') e^{-\epsilon |u-y'|}.$$
Thus,a simple computation shows that
\begin{align*}
l_\epsilon(\gamma) &= \int_\gamma \rho_\epsilon (u) |du| \\
&\leq e^{\epsilon 16\delta+\epsilon 10h} \rho_\epsilon(y') \int_\gamma e^{-\epsilon |u-y'|}|du|\\
&\leq e^{\epsilon 16\delta+\epsilon 10h} \rho_\epsilon(y')\left[ \int_{\gamma[x, y']} e^{-\epsilon l(\gamma[u, y'])+\epsilon h}|du| + \int_{\gamma[y', y] } e^{-\epsilon l(\gamma[u, y'])+\epsilon h}|du| \right]\\
&\leq 2e^{\epsilon 16\delta+\epsilon 11h} \rho_\epsilon(y') \int_0^\infty e^{-\epsilon t} dt\\
&\leq 2\epsilon^{-1} e^{\epsilon 16\delta+\epsilon 11h} \rho_\epsilon(y'),
\end{align*}
which gives the upper bound in \eqref{inter} when $|x-y|>1/2$. For the lower bound of \eqref{inter}, in view of \eqref{allu_jose_02_0016}, we find that
$$b(u)\leq b(y') +|u-y'|+10\delta$$
for all $u\in \gamma$, which ensures that $\rho_\epsilon(u)\geq \rho_\epsilon(y') e^{-\epsilon |u-y'|} e^{- 10 \epsilon \delta}$. Therefore,
\begin{align*}
l_\epsilon(\gamma) &= \int_{\gamma} \rho_\epsilon(u) |du|\\
&\geq  \rho_\epsilon(y')e^{- 10 \epsilon \delta} \int_\gamma e^{-\epsilon |u-y'|} |du|\\
& \geq  \rho_\epsilon(y')e^{- 10 \epsilon \delta} \left[ \int_{\gamma[x, y']} e^{-\epsilon l(\gamma[u, y'])}|du| + \int_{\gamma[y', y] } e^{-\epsilon l(\gamma[u, y'])}|du| \right]\\
&\geq  \rho_\epsilon(y')e^{- 10 \epsilon \delta} \int_0^{|x-y|/2} e^{-\epsilon t} dt \\
&\geq  \rho_\epsilon(y')e^{- 10 \epsilon \delta} \frac{1}{\epsilon} \left(1-e^{-1/4}\right) \\
&\geq \frac{1}{2} \rho_\epsilon(y')e^{- 10 \epsilon \delta} \frac{1}{\epsilon}.
\end{align*}
This completes the proof.
\end{proof}
\begin{corollary} \label{notwritten}
There exists a constant $M=M(\delta, \epsilon)$ such that 
\be \label{notwritteneq}
d_\epsilon (x, y) \asymp_M e^{-\epsilon(x|y)_b} \min \left\{1/2, \epsilon|x-y|\right\} \text{ for all } x,y\in X.
\ee
\end{corollary}
\begin{proof}
The upperbound follows easily from \eqref{comparelemmaeq}. For the lower bound, take $h=1/14$ and choose an $h-$short arc $\gamma :x\cra y$ satisfying $l(\gamma)\leq 2|x-y|$. Then by Theorem \ref{allu_jose_02_0002}, we have
$$l_\epsilon(\gamma) \leq K(\delta) d_\epsilon(x, y).$$
Combining this with the lower bound for $l_\epsilon(\gamma)$ in Lemma \ref{comparelemma}, we get the desired result. 
\end{proof}
\begin{proof}[Proof of Theorem \ref{allu_jose_02_0007}]
Let $X$ be an intrinsic $\delta-$hyperbolic space with atleast two elements in the boundary $\delta_\infty X$. 
We equip $X$ with the metric defined as in \eqref{allu_jose_02_0006} and we intend to prove that the space $X_\epsilon=(X, d_\epsilon)$ are unbounded uniform spaces.
Our first aim is to prove that $X_\epsilon$ is a rectifiably connected metric space. To this end, we  show that the identity map $\left(X, |x-y|\right)\rightarrow \left(X, d_\epsilon\right)$ is locally bilipschitz.  Let $w\in X$ be fixed and consider the open ball $B(w, 1)$. Then, for every $z\in B(w, 1)$, by virtue of \eqref{allu_jose_02_0018}, we obtain
$$e^{-\epsilon-10\epsilon\delta}\rho_\epsilon (w)\leq \rho_\epsilon (z) \leq e^{\epsilon+10\epsilon\delta}\rho_\epsilon(w).$$
Now, for $x, y\in B(w, 1)$, choose an  $\alpha$ in $X$ joining them satisfying $l(\gamma)\leq 2|x-y|$, then for any $u\in \alpha$ we have by \eqref{allu_jose_02_0018} that,
$$\rho_\epsilon(u) \leq e^{10\epsilon\delta}e^{\epsilon|x-u|}\rho_\epsilon(x)\leq  e^{10\epsilon\delta}e^{4\epsilon}\rho_\epsilon(x) \leq e^{5\epsilon+20\epsilon\delta} \rho_\epsilon (w).$$
Thus, for $x, y\in B(w, 1)$, this implies that,
$$d_\epsilon(x, y) \leq \int_\alpha \rho_\epsilon \ds \leq 2e^{5\epsilon+20\epsilon\delta}\rho_\epsilon(w)|x-y| . $$
Now, for any curve $\gamma$ joining $x$ and $y$ we can find a subcurve $\gamma'$ of $\gamma$ emanating from $x$ such that $l(\gamma')=|x-y|$. Therefore, by \eqref{allu_jose_02_0018} one computes,
\begin{align*}
d_\epsilon(x, y) &= \inf_\gamma\int_\gamma \rho_\epsilon \ds \geq \inf_\gamma\int_{\gamma'} \rho_\epsilon \ds\\
&\geq \int_0^{|x-y|}\rho_\epsilon(x) e^{-10\epsilon \delta} e^{-\epsilon t}\dt\\
&\geq \rho_\epsilon(x) e^{-10\epsilon \delta-2\epsilon}|x-y|\\
&\geq e^{-20\epsilon\delta-3\epsilon}\rho_\epsilon(w)|x-y| .
\end{align*}
Hence, we see that the identity map is locally bilipschitz and thus $X_\epsilon$ is a rectifiably connected metric space. \\[2mm]
We next verify that $X_\epsilon$ is unbounded. 
For $h<1/13$, choose 
$h-$short arcs $\alpha_i:o\cra u_i$ satisfying $l(\alpha_i)\leq 2|o-u_i|$, where the sequence $\{u_i\}\in \omega$. 
By Lemma \ref{allu_jose_02_0012}, we may assume that for any $x\in X$ and for sufficiently large $n$, we have
$$(\omega|o)_x \leq (u_n|o)_x+\frac{1}{2}$$
and
$$(\omega|x)_o \geq (u_n|x)_o-\delta-\frac{1}{2}.$$
Therefore, for $b=b_{\omega, o}\in \mathcal{B}(\omega)$ and for sufficiently large $n$, we obtain,
$$b(x) = (\omega|o)_x-(\omega|x)_o \leq (u_n|o)_x-(u_n|x)_o+\delta+1= |u_n-x|-|u_n-o|+\delta+1.$$
In particular, for $x=\alpha_n(t)$, we have
$$b(\alpha_n(t))\leq |u_n-\alpha_n(t)|-|u_n-o|+\delta+1\leq h+\delta +1$$
for sufficiently large $n$.
Now, by combining this with Theorem \ref{allu_jose_02_0002}, we obtain,
$$d_\epsilon(0, u_n)\geq \frac{1}{M} l_\epsilon(\alpha_n) \geq \frac{1}{M} \int_{\alpha_n}e^{-\epsilon(\delta+h+1)}\dt \geq \frac{1}{M}e^{-\epsilon(\delta+2)}|o-u_n|$$
for sufficiently large $n$. Since, $|o-u_n|\ra \infty $ as $n\ra \infty$ we obtain that $X_\epsilon$ is unbounded.\\[2mm]
Next we prove that $X_\epsilon$ is incomplete as long as there exists $\zeta\neq \omega$ such that $\zeta\in \delta_\infty X$. 
Let $\{u_n\}\in \zeta $, we prove that  $\{u_n\} $ is a Cauchy sequence in $X_\epsilon$. 
From \eqref{3.6} we have
$$(u_n, u_m)_b \dot{=} (u_n, u_m)_o -(u_n|\omega)_o -(u_m|\omega).$$
Since  $\{u_n\}\in \zeta \neq \omega$ we have, $(u_n|\omega)_o <\infty $ and $(u_m|\omega)<\infty $ for every $n, m\in \mathbb{N}$. Therefore, we see that
$(u_n, u_m)_b \rightarrow \infty $ as $n, m \rightarrow \infty$. Let $\gamma : u_n \cra u_m $ be any $h-$short arc with $h<1/13$ satisfying $l(\gamma) \leq 2|u_n-u_m|$. Further, without loss of generality we assume that $\epsilon|u_n-u_m| \geq 1/2$. Then, by Lemma \ref{comparelemma}, we obtain that
$$d_\epsilon(u_n, u_m ) \leq l_\epsilon(\gamma) \leq \frac{C}{2} e^{-\epsilon (x|y)_b } \rightarrow 0  \text{ as } n, m\rightarrow \infty .$$
Thus,  $\{u_n\}$ is a Cauchy sequence in $(X_\epsilon, d_\epsilon)$.
Since, $\{u_n\}$ is a Gromov sequence in $X$, and $X$ and $X_\epsilon$ are locally bilipschitz we see that $\{u_n\}$ cannot converge to a point in $X_\epsilon$. Therefore, $X_\epsilon$ is not complete and hence $d_\epsilon X = \bar{X_\epsilon}\setminus X_\epsilon\neq \emptyset$.\\[2mm]
Using Harnack inequality, for all $x, y\in X$ satisfying $|x-y|>1/\epsilon$, we calculate 
\begin{align*}
d_\epsilon(x, y) &= \inf \int_\gamma \rho_\epsilon \ds\\
&\geq \int_0^{1/\epsilon} e^{-10\epsilon\delta} \rho_\epsilon(x) e^{-\epsilon t}\dt\\
&= \epsilon^{-1} e^{-10\epsilon\delta} \rho_\epsilon(x) \left(1-e^{-1}\right)\\
&\geq \epsilon^{-1} e^{-10\epsilon\delta -1} \rho_\epsilon(x)
\end{align*}
We write $d_\epsilon(x) = \dist_\epsilon (x, d_\epsilon X)$ for any $x\in X$. Then, we see that,
\be \label{bdrydist}
d_\epsilon(x) \geq \frac{\rho_\epsilon(x)}{\epsilon e^{10\epsilon\delta +1}}
\ee
It remains to prove that $X_\epsilon$ is uniform. Since, $X$ is intrinsic, any two points $x, y\in X$ can be joined by an $h-$short arc $\gamma$ satisfying $l(\gamma)\leq 2|x-y|$, where $0\leq h \leq 1/13.$ Therefore, by Theorem \ref{allu_jose_02_0002}, we observe that the quasiconvexity condition is satisfied. Thus, we need only verify the double cone arc condition for the curve $\gamma$.
By symmetry it suffices to prove the cone condition for all $z\in \gamma[z, y']$ where $y' \in \gamma $ is such that $l(\gamma[y', y]) = (x|\omega)_y$. \\[2mm]
To this end, by Lemma \ref{conelemma} we have
 $$b(u) \geq b(z)+ |u-z|-16\delta-10h, \text{ for all } u \in \gamma[x, z].$$
 Thus,
 \begin{align*}
 l_\epsilon(\gamma) &= \int_\gamma \rho_\epsilon(u) |du|\\
 &= \int_\gamma e^{-\epsilon b(u)} |du|\\
 & \leq \int_\gamma \exp\left\{-\epsilon\left[ b(z)+ |u-z|-16\delta-10h\right]\right\}|du|\\
 &\leq \rho_\epsilon(z) \exp \left\{\epsilon\left[16\delta+10h\right]\right\} \int_\gamma e^{-\epsilon|u-z|} |du|\\
 &\leq \rho_\epsilon(z) \exp \left\{\epsilon\left[16\delta+11h\right]\right\}\int_0^\infty e^{-\epsilon t} dt\\
 &\leq \exp \left\{\epsilon\left[26\delta+11h\right]+1\right\} d_\epsilon(z)
 \end{align*}
 where in the last step we used \eqref{bdrydist}. This completes the proof of the Theorem \ref{allu_jose_02_0007}.
\end{proof}
\begin{proof}[Proof of Theorem \ref{allu_jose_02_0007.1}]
 Let $\phi : \delta_\infty X \setminus \{\omega\} \rightarrow \delta_\epsilon X $ be defined as follows. For every $\zeta \neq \omega $ in $\delta_\infty X $, let $\{v_n\} \in \zeta $ then by the proof of incompleteness in Theorem \ref{allu_jose_02_0007}, we see that $\{v_n\}$ is a $d_\epsilon$ Cauchy sequence which converges to a limit on $\delta_\epsilon X$. 
We define $\phi(\zeta) $ as the limit of this Cauchy sequence in $\delta_\epsilon X$. \\[2mm]
Let $\{u_n\}, \{u_n'\} \in \eta \neq \omega$, then an application of \eqref{comparelemmaeq} shows that they are equivalent Cauchy sequences in the space $(X_\epsilon, d_\epsilon)$. Therefore, the map $\phi$ is well defined.\\[2mm]
Secondly we show that $\phi $ is injective. If for $\zeta , \eta \in \delta_\infty X \setminus \{\omega\}$ we have $\phi(\zeta) = \phi (\eta)$, then for any Gromov sequences $\{u_n\}\in \zeta$ and $\{v_n\} \in \eta $, we obtain that
$$d_\epsilon (u_n, v_n ) \rightarrow \infty \text{ as } n\rightarrow \infty.$$
Moreover, if $\epsilon |u_n-v_n| \leq 1/2$, we have
$$2(u_n|v_n)_o \geq (u_n|v_n)_o - |u_n-v_n| \rightarrow \infty \text{ as } n\rightarrow \infty,$$
which shows that $\{u_n\}$ and $\{v_n\}$ are equivalent Gromov sequences and hence $\zeta = \eta$. If $\epsilon |u_n-v_n| > 1/2$, then by Lemmas \ref{comparelemma}, \ref{allu_jose_02_0002} and from \eqref{3.6} we conclude that
$$(u_n|v_n)_o - (u_n|\omega)_o - (\omega|v_n)_o \dot{=} (u_n|v_n)_b \rightarrow \infty \text{ as } n\rightarrow \infty,$$
which implies that $(u_n|v_n)_o \rightarrow \infty$ and therefore $\zeta = \eta$. Hence, $\phi $ is injective.\\[2mm]
Let $a\in \delta_\epsilon X$ and let $\{a_n\} $ be a $d_\epsilon$ Cauchy sequence which converges to $a$. We prove that $\{a_n\} $ is a Gromov sequence in $X$.
Observe that $\{a_n\} $ cannot be Cauchy in $X$, because if it is indeed a Cauchy sequence then $\{a_n\} $ converges to some point in $X$. But since $X$ and $X_\epsilon$ are locally bilipchitz equivalent this would imply that $\{a_n\} $  is $d_\epsilon$ convergent to some point in $X_\epsilon$, which is a contradiction. Therefore, from Corollary \ref{notwritten}, we see that $(a_n|a_m)_b \rightarrow \infty $ as $n, m\rightarrow \infty$. Again, by \eqref{3.6}, we obtain
$$(a_n|a_m)_o - (a_n|\omega)_o- (\omega|a_m)_o \dot{=} (a_n|a_m)_b \rightarrow \infty \text{ as } n, m \rightarrow \infty ,$$
which guarantees us that $(a_n|a_m)_o \rightarrow \infty$. Hence,  $\{a_n\} $ is a Gromov sequence in $X$, and it is clear by the definition of $\phi$ that $a$ the image under $\phi$ of the Gromov boundary element represented by $\{a_n\} $. Thus $\phi$ is surjective. This completes the proof.
\end{proof}
\noindent{\bf Acknowledgement.}
The first author thanks SERB-CRG and the second author's research work is supported by CSIR-UGC.\\
\noindent\textbf{Compliance of Ethical Standards:}\\
\noindent\textbf{Conflict of interest.} The authors declare that there is no conflict  of interest regarding the publication of this paper.\vspace{1.5mm}\\
\noindent\textbf{Data availability statement.}  Data sharing is not applicable to this article as no datasets were generated or analyzed during the current study.\vspace{1.5mm}\\
\noindent\textbf{Authors contributions.} Both the authors have made equal contributions in reading, writing, and preparing the manuscript.

\end{document}